\documentclass[10pt]{article}
\usepackage{xspace}
\usepackage[english,french]{babel}
\usepackage[fixlanguage]{babelbib}
\selectbiblanguage{french}
\usepackage[latin1]{inputenc}
\usepackage[T1]{fontenc}
\usepackage{pslatex}
\usepackage{amsmath,amssymb,amsfonts,amsthm}
\usepackage{url}
\usepackage{fullpage}
\usepackage{hyperref}
\def\QQ{\ensuremath{\mathbb Q}}
\def\Qp{{\ensuremath{\QQ_p}}}
\def\ZZ{\ensuremath{\mathbb Z}}
\def\cH{\ensuremath{\mathcal H}}
\def\CC{\ensuremath{\mathbb C}}
\def\Cp{\ensuremath{\CC_p}}
\def\Zp{\ensuremath{\ZZ_p}}
\def\RR{\ensuremath{\mathbb R}}
\def\cR{\ensuremath{\mathcal R}}
\def\cB{\ensuremath{\mathcal B}}
\def\go{\ensuremath{\mathfrak o}}
\def\gO{\ensuremath{\mathfrak O}}
\def\l{\ell}
\DeclareMathOperator{\Fil}{Fil}
\def\cD{\ensuremath{\mathcal D}}
\def\Db#1#2{\cD_{\infty,#1}(#2)}
\def\Zi{H^1_{Iw}}
\def\Zd{H^2_{Iw}}
\def\cL{\ensuremath{\mathcal L}}
\def\cLV{\cL_{V,r^\ast}}

\def\Quo{\Db{f}{\cD}/\cLV\left(\cH(G_\infty) \otimes \Zi(K,V)\right)}
\def\norm#1#2{\mid\!\mid #1 \mid\!\mid_{#2}}
\def\abs#1#2{\mid #1 \mid_{#2}}
\def\hn#1{{\cH^{#1}}}
\DeclareMathOperator{\mydet}{det}
\theoremstyle{plain}
\newtheorem{thmb}{Théorème}
\newtheorem{propb}[thmb]{Proposition}

\title{Note sur les diviseurs élémentaires du régulateur d'Iwasawa}
\usepackage[noblocks]{authblk}

\author[*]{Bernadette Perrin-Riou}
\affil[*]{Université Paris-Saclay, CNRS, Laboratoire de mathématiques d'Orsay, 91405, Orsay, France.}

\date{\today}
\begin{document}
\maketitle
\selectlanguage{english}
\begin{abstract}
The article revisits a result of \cite{LLZ} concerning the structure
of the image by the Iwasawa regulator map
of the Iwasawa module associated with a semi-stable $p$-adic representation
on an unramified finite extension of $\Qp$ and gives a direct proof
based on the results of \cite{jams} in the crystalline case
and \cite{semi-stable} in the semi-stable case.
\end{abstract}
\selectlanguage{french}
\begin{abstract}
On reprend les résultats de \cite{LLZ}
concernant la structure de l'image par l'application régulateur d'Iwasawa
du module d'Iwasawa associée à une représentation $p$-adique
semi-stable d'une extension finie non ramifiée de $\Qp$
et on en donne une démonstration directe s'appuyant sur les résultats de \cite{jams}
dans le cas cristallin et de \cite{semi-stable} dans le cas semi-stable.
\end{abstract}

Soit $p$ un nombre premier impair,
$\QQ_{p,n}=\Qp(\mu_{p^n})$ le corps de définition sur $\Qp$ des racines $p^n$-ièmes
de l'unité pour $n\geq 0$, $\QQ_{p,\infty}=\Qp(\mu_{p^\infty})$ la réunion des
$\QQ_{p,n}$
et $G_\infty$ le groupe de Galois de $\QQ_{p,\infty}/\Qp$.
Rappelons que $G_\infty = \Gamma \times \Delta$ avec $\Delta$ d'ordre $p-1$ premier à p.
Soit $\chi$ le caractère cyclotomique : $G_\infty \to \ZZ_p^\times$
et $\omega$ la projection sur $\Delta$.
On fixe un générateur $\gamma$ de $G_\infty$ :
$\tilde{\gamma}=\langle\gamma\rangle$ est un générateur de $\Gamma$ et
$\gamma = \omega(\gamma) \tilde{\gamma}$.

Notons $\cH$ l'anneau des séries convergentes sur toute boule
$B(\rho) = \{x \in \Cp \;\textrm{ tel que } \;\abs{x}{p} \leq \rho\}$
pour $\rho < 1$. Ici, $\abs{\cdot}{p}$ est la valeur absolue sur $\Cp$
normalisée par $\abs{p}{p}=p^{-1}$.
Si $I$ est un intervalle de
$\RR^+$, on définit les espaces de fonctions analytiques:
\begin{equation*}
\hn{I}=\left\{\sum_{n\in \ZZ} a_n x^n \text{ t.q. $a_n\in \Qp$ et }
\lim_{n\to \pm\infty} |a_n|r^n =0 \text{ pour tout $r\in I$}
 \right\}\ .\end{equation*}
Pour $f\in \hn{I}$ et $\rho\in I$, on pose
$\norm{f}{\rho}=\sup_{\abs{x}{}=\rho} \abs{f(x)}{} =\sup_n\abs{a_n}{}\rho^n$.
On munit $\hn{I}$ de la topologie de convergence uniforme sur l'ensemble
des $x$ tels que
$\abs{x}{p} \in J$ pour $J$ sous-ensemble fermé de $I$.
%%($[r_1,r_2)$ étant considéré comme ouvert en $r_2$).
Les espaces $\hn{I}$ sont complets.
La réunion des $\hn{[\rho,1[}$ pour $\rho\in [0,1[$ munie de la topologie de la limite
inductive est l'anneau $\cR$ de Robba. L'anneau $\cB$ introduit dans \cite{semi-stable}
est le sous-espace de $\cR$ formé des éléments $f$ d'ordre fini, c'est-à-dire tels qu'il
existe un entier $r$ tel que
$\norm{p^{rn} f}{\rho^{1/p^n}}$ tend vers $0$ pour $\rho$ entre 0 et 1 dans
le domaine de convergence.

Introduisons comme dans \cite{semi-stable} les anneaux des polynômes en $\log x$ à coefficients
dans $\cR$ ou $\cB$.
Ici, $\log x $ est une branche du logarithme, c'est-à-dire n'importe quelle
fonction localement
analytique sur $\Cp - \{0\}$ vérifiant $\log xy = \log x + \log y $
et dont la dérivée en 1 est
1. On munit ces espaces de la topologie induite par la convergence
coefficient par coefficient.
Ces espaces sont munis d'un opérateur $\varphi$ induit par $x \to (1+x)^p-1$.
Soit $\psi$ le pseudo-inverse de $\varphi$ caractérisé par
$$\varphi \circ \psi (f) = p^{-1}\sum_{\zeta \in \mu_p} f( \zeta (1 + x)-1)\ .$$
L'anneau $\cH$ est à diviseurs élémentaires, ce qui signifie que
toute matrice admet une réduction diagonale :
pour toute matrice $A$ à coefficients dans $\cH$, il existe des matrices unimodulaires $P$
et $Q$ telles que $PAQ = diag (d_1, d_2, \cdots)$
avec $(d_{i}) \subset (d_{i+1})$ (\cite{kaplansky}, \cite{lls}).
Ainsi, si $M$ est un $\cH$-module de type fini et $N$ un sous-$\cH$-module de $M$
de type fini, il existe une base $E_1, \cdots, E_d$ de $M$ et des éléments
$f_1, \cdots, f_d$ de $\cH$ vérifiant $f_d \mid \cdots \mid f_i \mid \cdots \mid f_1$
et tels que $f_1E_1,\cdots, f_d E_d$ soit un système générateur de $N$.
On note $[M : N ]= [f_1; ... ; f_d] $, suite des diviseurs élémentaires définis à une unité près.

La transformée de Mellin-Amice induite par $\gamma \mapsto (1+x)^{\chi(\gamma)}$
définit une application $\Zp[[G_\infty]] \to \Zp[[x]]$ :
$f\mapsto f\cdot (1+x)$ qui se prolonge en un homomorphisme de $\cH(G_\infty)$
dans $\cH$ dont l'image est $\cH^{\psi=0}$.
L'image de $\cH(\Gamma)$ de $\cH$ est égale à $\varphi(\cH)(1+x)$.
Le groupe $\Delta$ permute les $\varphi(\cH)(1+x)^i$ avec $i=0, \cdots, p-1$.
Si $M$ est un $\Zp[[G_\infty]]$-module et $\delta$ un caractère de $\Delta$ dans
$\ZZ_p^\times$, $M^{\delta}$ désigne la composante isotypique de $M$ relatif à $\delta$
et $e_\delta$ le projecteur associé de $M$ dans $M^\delta$.

On pose $\omega_n=(1+x)^{p^n} -1$ pour $n \geq 0$,
$\xi_n = \omega_n/\omega_{n-1} = \sum_{i=0}^{p-1} (1+x)^{i p^{n-1}}$ pour $n \geq 1$.
On fixe une unité $p$-adique $u=\chi(\tilde{\gamma})$ générateur de $1+p\Zp$.
Cela induit un unique homomorphisme d'anneaux $\chi_u : \Zp[[x]] \to \Zp$ tel que
$u =\chi_u(1+x)$.
Si $j$ est un entier et si $f \in \cH$, on note $Tw_u^{-j}(f) = f(u^{-j}(1+x)-1)$ ou
plus simplement $f^{(j)}$.
Ainsi, $\omega_n^{(j)} =u^{-jp^n}(1+x)^{p^n}-1$.
On pose
$$\l_{j}= \frac{\log u^{-j}\gamma}{\log u}=\frac{\log \gamma}{\log u} -j
\in \cH(G_\infty)\ .$$
On a donc $\chi^k (\l_{j}) =k-j$, $\chi^j (\l_{j}) =0$.
Vu comme élément de $\cH$, on a
$$\l_{j}= \frac{\log u^{-j}(1+x)}{\log u}=\frac{\log (1+x)}{\log u} -j
=Tw_u^{-j}(\log(1+x))\in \cH\ .$$
La restriction $\l_{j,\rho}$ de $\l_{j}$ à la couronne
$C_\rho=\{x\in \Cp \;\textrm{ tel que }\; \abs{x}{p} = \rho\}$
est une unité dans
$\hn{[\rho,\rho]}$ sauf si $\rho$ est de la forme $p^{-\frac{1}{p^{n-1}(p-1)}}$ avec $n\geq 1$.
Dans ce cas, $\l_{j,\rho}$ est irréductible et égal à une unité près à $\xi_n^{(j)}/p$.

Soit $K$ une extension finie non ramifiée de $\Qp$ et $\sigma$ l'homomorphisme de Frobenius associé
agissant sur $K$.
Si $\cD$ est un $\varphi$-module de dimension finie sur $K$ (c'est-à-dire un
espace vectoriel sur $K$ muni d'un opérateur $\varphi$ $\sigma$-linéaire) et
si $g \in \cH \otimes_{\Qp} \cD$, on définit $\go_\varphi(g)$ (resp.
$\gO_\varphi(g)$) comme la borne inférieure des réels $r$ tels que la suite
$\norm{p^{rn}(1\otimes \varphi)^{-n}g}{\rho_n}$ tend vers $0$ (resp. est
bornée) lorsque $n$ tend vers l'infini (avec $\rho_n=\rho^{1/p^n}$).
Cette définition s'étend à $\cB[\log(x)] \otimes \cD$ en considérant que
$\norm{(1\otimes \varphi)^{-n}\log(x)}{\rho_n} \asymp_{n\to\infty} p^{-n}$.

Soit $V$ une représentation $p$-adique semi-stable du groupe de Galois absolu de $G_K$,
de dimension finie $d$ et soit
$\cD$ le $(\varphi,N)$-module filtré associé par la théorie de Fontaine :
c'est donc un $K$-espace vectoriel de dimension $d$ muni d'un
opérateur $\varphi$, d'un opérateur $N$ nilpotent tel que $N\varphi = p \varphi N$ et
d'une filtration décroissante exhaustive et séparée $\Fil^\bullet$.
Notons $r$ un entier tel que $\cD = \Fil^{-r} \cD$
 et $r^\ast$ un entier tel que $\Fil^{r^\ast+1} \cD = 0$
(si $r^\ast=0$, les sauts de la filtration sont donc négatifs ou nuls).
Il est naturel de prendre pour $r$
(resp. $r^\ast$) le plus petit entier vérifiant $\cD = \Fil^{-r} \cD$
(resp. $\Fil^{r^\ast} \cD \neq 0$).
Notons $-r_d \leq \cdots \leq -r_i \leq \cdots \leq -r_1 $
les poids de la filtration $\cD$ comptés avec multiplicité : on a
$-r \leq -r_d$ et $-r_1 \leq r^\ast$.
Si $k$ est un poids de la filtration ($\Fil^{k}\cD \neq \Fil^{k+1}\cD)$,
la dimension de $\Fil^{k}\cD$ sur $K$ est égale au nombre d'entiers $i$ tels
que $k \leq -r_i$ et la dimension de $\Fil^{k}\cD/\Fil^{k+1}\cD$ est égale
au nombre d'entiers $i$ tels que $k=-r_i$.

Dans \cite{semi-stable}, sont associés à $\cD$
des $\cH(G_\infty)$-modules $\Db{e}{\cD} \subset \Db{f}{\cD} \subset \Db{g}{\cD}$
contenus dans $(\cB[\log x] ^{\psi=0} \otimes \cD)^{N=0}$,
de rang $d$ ($\Delta$-composante par $\Delta$-composante)
et munis d'une action de $\varphi$ induite par $1 \otimes \varphi$
(le fait que $\cD$ est relié à une représentation $p$-adique ne sert pas dans ces définitions).
Lorsque l'opérateur $N$ est nul sur $\cD$, $\Db{f}{\cD}$ est égal à
$\cH^{\psi=0} \otimes \cD$ et donc par la transformée d'Amice isomorphe à
$\cH(G_\infty) \otimes \cD$. Dans le cas où $N$ est non nul,
$\Db{f}{\cD}$ que nous ne redéfinissons pas est encore isomorphe\footnote{
Il est démontré dans \cite{semi-stable} le résultat suivant
que nous énonçons pour simplifier dans le cas où aucune valeur propre de $\varphi$
n'est une puissance de $p$.
Soit $\Theta_k = \lim_{n\to \infty} p^{kn}\psi^{n}(\log^k x) \in \cB[\log x]$ et
soit $\Theta$ l'opérateur de $\cB[\log x] \otimes \cD$ défini par
$\Theta=\sum_{k=0}^\infty \frac{(-1)^k}{k!} \Theta_k \cdot (1\otimes N)^k
=\lim_{n\to \infty}\psi^n \exp\left(-p^n \log(x) \cdot (1\otimes N)\right)$.
Il existe un unique isomorphisme $ \cH(G_\infty) \otimes \cD \to \Db{f}{\cD}$
défini (modulo les identifications du type transformée d'Amice) par
$(1-\varphi\otimes \varphi) \circ \Theta \circ (1-\varphi\otimes \varphi)^{-1}$.
}
à $\cH^{\psi=0} \otimes \cD$.

Soit $\Zi(K,V)$ le module d'Iwasawa
associé à $V$ : si $T$ est un réseau de $V$ stable par $G_\Qp$, $\Zi(K,V)$
est la limite projective des groupes de cohomologie $H^1(\QQ_{p,n}, T)$ tensorisé par $\Qp$.
Le module de torsion de $\Zi(K,V)$ est isomorphe à $V^{G_{K_\infty}}$
et $\Zd(K,V)$ est isomorphe à $V^\ast(1)^{G_{K_\infty}}$ par la dualité de Tate.
Soit $\Omega_{V,r}$ l'homomorphisme de $\cH(G_\infty)$-modules
$$\Omega_{V,r} : \Db{f}{\cD} \to \cH(G_\infty)\otimes \Zi(K,V) /V^{G_{\QQ_\infty}} $$
défini dans \cite{bpr-debut} dans le cas cristallin et dans \cite{semi-stable}
dans le cas semi-stable.

Les deux théorèmes suivants sont une conséquence de la loi de réciprocité
démontrée par Colmez (\cite{colmez}, voir aussi \cite[\S 5]{bpr-colmez} dans le cas cristallin
et \cite[5.3.6]{semi-stable} dans le cas semi-stable) et sont démontrés dans \cite[Th. 2.5.2]{jams})
(cas cristallin) et dans \cite[Th. B2 et Prop. 5.4.5]{semi-stable}) :
\begin{thmb}
Notons $\delta_{r}(\Omega_V)$ le déterminant de $\Omega_{V,r}$ calculé dans
une base du $\cH(G_\infty)$-module
$$\mydet_{\cH(G_\infty)} \Db{f}{\cD} \otimes
\left(\mydet_{\Qp\otimes\Lambda} \Zi(K,V)\right)^{-1} \otimes \mydet_{\Qp\otimes\Lambda} \Zd(K,V)\ .$$
Alors
$$ \prod_{-r^\ast \leq j < r} \l_{j} ^{\dim_{\Qp}\Fil^{-j} \cD} \delta_{r}(\Omega_V) = \cH(G_\infty)\; .$$
\end{thmb}

\begin{thmb}
\label{defg}Soit $x\in
\cH(G_\infty)\otimes \Zi(K,V)$. Il existe un unique élément $\cLV(x)$ de $\Db{f}{\cD}$ tel que
\[ \Omega_{V,r}\left(\cLV(x)\right)= \prod_{-r^\ast \leq j < r }\l_{j}.x \bmod torsion\]
et on a
$\go_\varphi\left(\cLV(x)\right)=\go(x)+r^*$ et
$\gO_\varphi\left(\cLV(x)\right)=\gO(x)+r^*$.
\end{thmb}
Il est facile de vérifier que $\cLV(x)$ ne dépend pas de $r$ vérifiant $\Fil^{-r}\cD =\cD$
et que
$$\cL_{V,r^\ast+s}(x) = \prod_{-r^\ast-s \leq j < -r ^\ast} \l_{j} \cLV(x)\; .$$
On a ainsi défini un $\cH(G_\infty)$-morphisme
$$\cH(G_\infty)\otimes \Zi(K,V) \to \Db{f}{\cD}\; .$$
Lorsque $V^{G_{K_\infty}}=0$,
$\Zd(K,V)$ est nul car isomorphe à $V^\ast(1)^{G_{K_\infty}}=0$ et
$\Zi(K,V)$ n'a pas de $\Qp\otimes \Lambda$-torsion.
\begin{propb}Si $V^{G_{K_\infty}}=0$,
l'annulateur du $\cH(G_\infty)$-module $\Quo$ est l'idéal engendré par
$\prod_{ -r^\ast \leq j < r_d}\l_{j}$.
\end{propb}
On peut voir cette proposition comme une propriété de "semi-simplicité" de $\cLV$ :
l'annulateur n'a que des facteurs simples.

\begin{proof}
On déduit du théorème \ref{defg}
que si $g \in \Db{f}{\cD}$, $\left(\prod_{ -r^\ast \leq j < r_d}\l_{j} \right) g$ appartient
à l'image de $\cH(G_\infty) \otimes \Zi(K,V)$ par $\cLV$.
En effet, on a
\[ \Omega_{V,r_d}\left(\cLV(\Omega_{V,r_d}(g))\right) =
\left(\prod_{ -r^\ast \leq j < r_d}\l_{j}\right) \Omega_{V,r_d}(g)\]
et donc par injectivité de $\Omega_{V,r_d}$,
$$
\left(\prod_{-r^\ast \leq j < r_d}\l_{j}\right) g =\cLV\left(\Omega_{V,r_d}(g))\right)
\in \cLV\left(\cH(G_\infty) \otimes \Zi(K,V)\right)
\; .$$
Ainsi, $\prod_{ -r^\ast \leq j < r_d}\l_{j}$ annule $\Quo$.

Montrons maintenant que $\prod_{-r^\ast \leq j < r_d}\l_{j}$ est l'exposant
de $\Quo$. Pour cela, on travaille composante par composante de
$\cH(G_\infty) = \oplus_{\delta \in \widehat{\Delta}} e_\delta\cH(\Gamma)$.
Soit $\delta$ un caractère de $\Delta$ et
soit $F$ un élément de $\cH(\Gamma)$ annulant
$$e_\delta \left(\Quo\right)\ .$$
En prenant le pgcd de $F$ et
de $\prod_{-r^\ast \leq j < r_d}\l_{j}$ (on utilise une identité de Bezout),
on peut supposer que $F$ est
un diviseur de $\prod_{-r^\ast \leq j < r_d}\l_{j}$. Notons
$Q$ le quotient de $\prod_{-r^\ast \leq j < r_d}\l_{j}$ par $F$.
Soit $y$ un élément de $e_\delta \Db{f}{\cD}$. Il existe
$x \in e_\delta \cH(G_\infty) \otimes \Zi(K,V)$ tel que
$ F y = \cLV(x)$. Par définition de $\cLV$,
$F \Omega_{V,r_d}(y) = \left(\prod_{-r^\ast \leq j < r_d}\l_{j}\right) x$ et
$\Omega_{V,r_d }(y) = Q x$. Donc, l'image de $e_\delta \Omega_{V,r_d}$ est contenue
dans $Q \cH(\Gamma) \otimes e_\delta \Zi(K,V)$. On a
$$e_\delta \cH(G_\infty) =
\prod_{-r^\ast \leq j < r_d}\l_{j}^{-\dim_{\Qp} \Fil^{-j} \cD} e_\delta \mydet \Omega_{V,r}
\subset
\prod_{-r^\ast \leq j < r_d}\l_{j}^{-\dim_{\Qp} \Fil^{-j} \cD}
  Q^{d[K:\Qp]} e_\delta\cH(G_\infty)\ ,$$
d'où l'inclusion
$$\left(\prod_{-r^\ast \leq j < r_d}\l_{j}^{\dim_{\Qp} \Fil^{-j} \cD}\right)e_\delta \cH(G_\infty)
\subset Q^{d[K:\Qp]} e_\delta \cH(G_\infty)\; .$$
On en déduit que si $R$ est un diviseur irréductible de $Q$,
donc divisant un des $\l_{j}$ pour un entier $j$ tel que
$-r^\ast \leq j<r_d$, on a nécessairement
$\dim \Fil^{-j} \cD\geq \dim \cD$, ce qui est impossible
puisque $\Fil^{-j} \cD \neq \cD$ pour un tel $j$. Donc $Q$ est une unité dans $e_\delta \cH(G_\infty)$.
\end{proof}

\begin{propb}Supposons $V^{G_{K_\infty}}=0$.
Le déterminant de $\cLV$ calculé dans des bases de $\Db{f}{\cD}$
et de $\cH(G_\infty) \otimes \Zi(K,V)$ est le $\cH(G_\infty)$-idéal engendré par
$$\prod_{-r^\ast \leq j < r}\l_{j} ^{[K:\Qp](\dim\cD-\dim \Fil^{-j}\cD)}
\; .$$
\end{propb}
\begin{proof}
La proposition se déduit de l'affirmation sur le déterminant de $\Omega_{V,r}$ :
$$\prod_{-r^\ast \leq j < r } \l_{j} ^{\dim\cD}
\prod_{-r^\ast \leq j < r} \l_{j} ^{-\dim \Fil^{-j} \cD} =
\prod_{-r^\ast \leq j < r}\l_{j} ^{\dim\cD-\dim \Fil^{-j} \cD}
=\prod_{-r^\ast \leq j < r} \l_{j} ^{\#\{i | j < r_i\}}
\; .$$
\end{proof}

\begin{thmb}On suppose que $\cD^{\varphi=p^j} = 0 $ pour tout entier $j\in \ZZ$.
La suite des diviseurs élémentaires des $e_\delta\cH(G_\infty)$-modules
$$e_\delta\cLV\left(\cH(G_\infty) \otimes \Zi(K,V)\right) \subset e_\delta\Db{f}{\cD} $$
est
$$\left[\prod_{ -r^\ast \leq j < r_d}\l_{j} \; ; \; \cdots \ ; \;
\prod_{-r^\ast \leq j < r_i }\l_{j}\; ;\
\cdots \; ; \; \prod_{-r^\ast \leq j < r_1}\l_{j}\right]\ ,$$
chaque terme étant répété $[K:\Qp]$ fois.
\end{thmb}
\begin{proof}
Vérifions que le résultat est compatible avec la valeur du déterminant.
Le produit des $\prod_{-r^\ast \leq j < r_i}\l_{j}$ est égal à
\begin{equation*}
\prod_{i=1}^d \prod_{-r^\ast \leq j < r_i} \l_{j}=
\prod_{-r^\ast \leq j < r_d} \l_{j}^{\#\{i \mid j < r_i\}} =
\prod_{-r^\ast \leq j < r_d}\l_{j}^{\dim\cD/\Fil^{-j} \cD}
=\prod_{-r^\ast \leq j}\l_{j}^{\dim\cD/\Fil^{-j} \cD}\ ,
\end{equation*}
et $\prod_{i=1}^{r_d}\left(\prod_{-r^\ast \leq j < r_i}\l_{j}\right)^{[K:\Qp]}$
est bien égal au déterminant de $\cLV$.
Posons $\tilde{d}=[K:\Qp]d$.
Notons $[f_1; \cdots; f_{\tilde{d}}]$ la suite des diviseurs élémentaires de
$e_\delta\cLV\left(\cH(G_\infty\right) \otimes \Zi(K,V)) \subset e_\delta\Db{f}{\cD}$.
Le générateur $\prod_{-r^\ast \leq j < r_d}\l_{j}$ de son exposant
est égal à $f_1$. Soit $0 \leq \rho <1$.
Comme les $f_i$ divisent $f_1$, on peut écrire dans la couronne $C_\rho$
 $$f_{1,\rho}=u_1\prod_{ -r^\ast \leq j < r_d}
 \l_{j,\rho}^{\alpha_{j,1,\rho}} \quad, \cdots, \quad
 f_{\tilde{d},\rho} = u_{\tilde{d}}\prod_{-r^\ast \leq j < r_d} \l_{j,\rho} ^{\alpha_{j,r,\rho}}$$
avec $(\alpha_{j,s,\rho})_s$ des suites décroissantes d'entiers égaux à 0 ou 1 et
$u_s$ des unités dans $C_\rho$.
Les $\l_{j,\rho}$ sont des unités sur $C_\rho$ sauf si
$\rho=p^{-\frac{1}{(p-1)p^{n-1}}}$ pour un entier $n\geq 1$ ; dans ce cas, $\l_{j,\rho}$ est égal à une unité près
à $\xi_n(u^{-j}(1+x)-1)$ qui est $\rho$-dominant (les zéros sont
les $u^{j} \zeta -1$ pour $\zeta$ racine de l'unité d'ordre $p^n$).
Le nombre $n_{j,\rho}$ d'entiers $s$ tels que $\alpha_{j,s,\rho}=1$ suffit à déterminer
les $\alpha_{j,1,\rho}, \cdots, \alpha_{j,r,\rho}$.
En comparant à la valeur du déterminant, on a (à une unité près)
$$\prod_{s=1}^{\tilde{d}} f_{s,\rho}
\sim \prod_{-r^\ast \leq j < r_d} \l_{j,\rho}^{n_{j,\rho}}
\sim \prod_{-r^\ast \leq j < r_d}\l_{j,\rho}^{\dim_{\Qp}\cD/\Fil^{-j} \cD}\; .$$
Les $n_{j,\rho}$ sont donc bien déterminés :
$n_{j,\rho} =\dim_{\Qp} \cD/\Fil^{-j}\cD$, ainsi que les $\alpha_{j,s,\rho}$
(en particulier, ils ne dépendent pas de $\rho$).
On en déduit le théorème.
\end{proof}

\bibliographystyle{amsplain}
\selectbiblanguage{french}
\bibliography{biblio}
\end{document}